\documentclass[a4paper,10pt]{amsart}
\usepackage[T1]{fontenc}
\usepackage[utf8]{inputenc}
\usepackage{amsmath,amsthm,amsfonts,amssymb,latexsym}
\usepackage{amsmath}

\def\la#1{\hbox to #1pc{\leftarrowfill}}
\def\ra#1{\hbox to #1pc{\rightarrowfill}}
\def\fract#1#2{\raise4pt\hbox{$ #1 \atop #2 $}}
\def\bbg{\mathcal{R}e}

\def\bbS{{S}}
\def\bbs{{ S^{\ast}}}
\def\H2{\hbox{{\rm H}\kern-0.9em\hbox{{\rm I}\ \ }}^2}
\def\h{\hbox{{\rm H}\kern-0.9em\hbox{{\rm I}\ \ }}}
\def\T{\hbox{{\rm T}\kern-0.66em\hbox{{\rm I}\ \ }}}
\def\R{\hbox { {\rm R}\kern-0.9em\hbox{{\rm I}\ }}}
\def\C{\hbox { {\rm C}\kern-0.56em\hbox{{\rm I}\ }}}

\theoremstyle{definition}
\theoremstyle{remark}
\numberwithin{equation}{section}

\newtheorem{theorem}{\textbf{Theorem}}[section]
\newtheorem{corollary}[theorem]{\textbf{Corollary}}
\newtheorem{lemma}[theorem]{\textbf{Lemma}}

\newtheorem{proposition}[theorem]{\textbf{Proposition}}

\newtheorem{remark}{\textbf{Remark}}[section]

\begin{document}
\title{On the numerical radius of the truncated adjoint Shift}
\author{Haykel GAAYA}
\address{\ddag.Institute Camille Jordan, Office 107 University of Lyon1,
43 Bd November 11, 1918, 69622-Villeurbanne, France.}
\email{\ddag  gaaya@math.univ-lyon1.fr }
\maketitle
\begin{abstract}
A celebrated thorem of Fejer (1915) asserts that for a given positive trigonometric polynomial
 $\sum_{j=-n+1}^{n-1}c_{j}e^{ijt}$, we have $\lvert c_{1}\lvert\leqslant c_{0}\cos\frac{\pi}{n+1}$. A more recent inequality due to U. Haagerup and P. de la Harpe \cite{Haagerup} asserts that, for any contraction $T$ such that $T^{n}=0$, for some $n\geq2$, the inequality $\omega_{2}(T)\leqslant\cos\frac{\pi}{n+1}$ holds, and $\omega_{2}(T)=\cos\frac{\pi}{n+1}$ when T is unitarily equivalent to the extremal operator ${\bbS}^{\ast}_{n}={\bbs}_{\lvert{\C}^{n}}={\bbs}_{\lvert Ker (u_{n}(\bbs))}$ where $u_{n}(z)=z^{n}$ and $\bbs$ is the adjoint of the shift operator on the Hilbert space of all square summable sequences. Apparently there is no relationship between them. In this mathematical note, we show that there is a connection between Taylor coefficients of positive rational functions on the torus and numerical radius of the extremal operator $\bbs(\phi)=\bbs_{\lvert Ker(\phi(\bbs))}$ for a precise inner function $\phi$. This result completes a line of investigation begun in 2002 by C. Badea and G. Cassier \cite{Cassier}. An upper and lower bound of the numerical radius of $\bbs(\phi)$ are given where $\phi$ is a finite Blashke product with unique zero.
\end{abstract}

\section{Introduction}
\hspace{0.5cm}Let $\mathcal{H}$ be a complex separable Hilbert space and $\mathcal{B(H)}$ the collection of all
 bounded linear operators on $\mathcal{H}$. The numerical range of an operators $T$ in $\mathcal{B(H)}$ is the subset $$W(T)=\left\lbrace <Tx,x>\in\C;x\in \mathcal{H} ,\lVert x \lVert= 1\right\rbrace $$  of the plane, where $<.,.>$ denotes the inner product in $\mathcal{H}$ and the numerical radius of $T$ is defined by $$ \omega_{2}(T)=\sup \left\lbrace \lvert z\lvert; z\in W(T) \right\rbrace .$$
$\mathcal{R}e(T)$ is the self-adjoint operator defined by $$\bbg(T)=\frac{1}{2}(T+T^{\ast}).$$ We denote by $\left[ x\right]$ the integer part of $x$ and by  $\bbS$ the unilateral shift acting on the Hardy space $\H2$ of the square summable analytic functions and by $\bbs$ its adjoint:
$$\begin{array}{ccccc}
\bbS & : & \H2 & \to & \H2 \\
& & f & \mapsto & zf(z) \\
\end{array}$$

$$\begin{array}{ccccc}
\bbs & : & \H2 & \to & \H2 \\
& & f & \mapsto & \dfrac{f(z)-f(0)}{z} \\
\end{array}.$$

\hspace{0.5cm}Beurling's theorem implies that the non zero invariant subspaces of $\bbS$ are of the forme $\phi~\H2$, where $\phi$ is some inner function . Let $\bbS(\phi)$ denote the compression of $\bbS$ to the space $ H(\phi)=\H2\ominus\phi~\H2 $ :
$$\bbS(\phi)f(z)=P(zf(z)),$$ where $P$ denotes the ortogonal projection from $\H2$ onto $ H(\phi)$. We denote by $\bbs(\phi)$ the adjoint of $\bbS(\phi)$: $$\bbs(\phi)=\bbS(\phi)^{\ast}={\bbs}_{\lvert H(\phi)}=\bbs_{\lvert Ker(\phi(\bbS)^{\ast})} ~.$$ 

\hspace{0.5cm}The model operator $\bbS(\phi)$ has many properties and it was studied intensively in the 1960s and '70s. For exemple, it has norm 1 (for dim  $H(\phi)>1$) and it is cyclic. The function $\phi$ is the minimal function of $\bbS(\phi)$ meaning that $\phi(\bbS(\phi))=0$ and $\phi$ divides any function $\psi$ in $H^{\infty}$ with $\psi(\bbS(\phi))=0$. The space $H(\phi)$ is  finite-dimensional exactly when $\phi$ is a finite Blaschke product : $$\phi(z)=\prod_{j=1}^{n}\dfrac{z-\alpha_{j}}{1-\overline{\alpha_{j}}z}.$$
In this case the polynomial $p(z)=\prod_{j=1}^{n}(z-\alpha_{j})$ is both the minimal and characteristic polynomial of $\bbS(\phi)$ and $(\alpha_{j})_{1\leqslant j\leqslant n}$ are its eigenvalues. In particular, if $\phi(z)=z^{n}$ then $\bbS(\phi)$ is unitarily equivalent to $\bbS_{n}$ where 

$$
\bbS_{n}=\left(
\begin{array}{cccc}
0 &       &       &  \\
1 & \ddots &      &   \\
  & \ddots &\ddots    &  \\
  &        & 1    & 0 
  
\end{array}
\right)
.$$
\hspace{0.5cm}In 1992 U. Haagerup and P. de la Harpe proved that $W(\bbS_{n})$ is 
the disc\\ $D_{n }=\left\lbrace z\in \C ; \lvert z\lvert \leqslant \cos\frac{\pi}{n+1}\right\rbrace$ and 
\begin{equation}
\omega_{2}(\bbS_{n})=\cos\frac{\pi}{n+1}
\end{equation}
and more generally a natural connection between Fejer's inequality and the numerical radius of a nilpotent matrix was established by Haagerup and de la Harpe. They proved, using solely elementary methods (positive definite kernels) that:
\begin{theorem}[\cite{Haagerup}]
\textit{ Let $T$ be an operator on $\mathcal{H}$ such that $T^{n}=0$ for some $n \geq 2$. One has:$$\omega_{2}(T)\leqslant \lVert T\lVert \cos\frac{\pi}{n+1}$$
and $\omega_{2}(T)=\lVert T\lVert \cos\frac{\pi}{n+1}$ when $T$ is unitarily equivalent to $\lVert T\lVert \bbS_{n}$}.
\end{theorem}

\ \hspace{1cm}The reader may consult \cite{Halmos} chapter 22 for properties of numerical ranges of operators in general, \cite{Horn} chapter 1 for those of finite dimensional operators and particulary \cite{Wu2} for the geometric properties of the numerical range of $\bbS(\phi)$.

\underline{Organisation of the paper:}

In \cite{Cassier}, C. Badea and G. Cassier showed: 
\begin{theorem}(\cite{Cassier})
 \textit{Let $F=P/Q$ be a rational function with no principal part $(d^{\circ}P < d^{\circ}Q)$ which is positive on the torus. 
Then the Taylor coefficient $c_{k}$ of order $k$ satisfies the following inequalitiy $$\lvert c_{k}\lvert\leqslant c_{0}~\omega_{2}(R^{k}),$$ where $R={\bbs}_{\lvert Ker(Q(\bbs))}$.}
\end{theorem}
\hspace{0.5cm}In Section 2, our main theorem is the Theorem 1.2. We give an extension of the result of C. Badea and G. Cassier 
for Taylor coefficients of all rational functions which are positive on the torus. We make no extra assumptions about 
$P$ and $Q$. We do not, for exemple require them to obey any degree restrictions, they need only be coprime. The theorem has 
many applications and will explain how we can easily recover the remarkable Egerváry and Szász inequality. See Corollary 2.3.

\hspace{0.5cm}Toeplitz matrices are found in several areas of mathematics and its applications such as complex and harmonic analysis. 
The KMS Toeplitz matrix 
$$K_{n}(\alpha)=(\alpha^{\lvert r-s\lvert})^{n}_{r,s=1}$$ associated with the Poisson kernel introduced by Kac, Murdokh and Szegö \cite{Kac} is of particular interst in these areas. In order to formulate our problem we first review in Section 3.1 some of the known results on the spectra of these matrices and we give a better upper bound of $\lambda_{1}^{(n)}$, where $\lambda_{1}^{(n)}$ is the largest eigenvalue of the KMS matrix. In Section 3.2, our main theorem is the Theorem 3.7 which gives un upper and lower bound of the numerical radius of the truncated shift $\bbs(\phi)$ where $\phi(z)=\left( \dfrac{z-\alpha}{1-\overline{\alpha}z}\right)^n $ with $0\leqslant\alpha<1$ is a finite Blashke product with unique zero. Our preoccupation with this particular case $0\leqslant\alpha<1$ to the exclusion to any $\alpha$ in the disc is explained by the fact that the numerical radius is independent with the argument of $\alpha$. This formula is expressed in terms of eigenvalues of the KMS matrices and provides an easy proof for the Haagerup and de la Harpe result (1.1). 
\section{Main Theorem}
\hspace{0.5cm}There are many classical inequalities for coefficients of positive trigonometric polynomials. 
The next result shows the links between the numerical radius of the truncated adjoint shift and the Taylor coefficients of rational functions positive on the torus.
\begin{theorem}
\textit{Let $F=P/Q$ be a rational function which is positive on the torus, where ${P}$ and ${Q}$ are coprime. Denote by $$\phi(z)=\prod_{j=1}^{p}\left( \dfrac{z-\alpha_{j}}{1-\overline{\alpha_{j}}z}\right)^{m_{j}} $$ and $$\psi(z)=\prod_{j=1}^{q}\left( \dfrac{z-\beta_{j}}{1-\overline{\beta_{j}}z}\right)^{d_{j}} $$ the respectively finite Blashke products formed by the nonzero roots of $P$ and $Q$ in the open disc, let $m=\sum_{j=1}^{p}m_{j}$ and $d=\sum_{j=1}^{q}d_{j}$. Then the Taylor coefficient $c_{k}$ of order $k$ of $F$ satisfies the following inequality:
 \begin{eqnarray*}
  \lvert c_{k} \lvert \leqslant c_{0}~\omega_{2} ({\bbs}^{k}(\varphi)),~ \mbox{where}~ \varphi(z)=z^{\max(0,m-d+1)}\psi(z).
 \end{eqnarray*}
}
\end{theorem}
\begin{lemma}[\cite{Cassier} lemma 3.2] 
\textit{Let $u$ be a inner function and let $f$ be a positive function in the subspace $\overline{u}~{\h}_{0}^{1}$ of $L^{1}(\T)$. Then there exists a function $h$ in $H(u)=\H2\ominus u~\H2$ such that $f=\lvert h\lvert^{2}$.}
\end{lemma}

\begin{proof}
First, note that by continuity we may assume that $F$ is strictly positive on the torus. Let $F=P/Q$ and assume that $F(z)>0$ for every $z\in \T$. Now, let  $$G\left(z\right)=\overline{F\left( \frac{1}{\overline{z}}\right) } .$$ We see that $G$ is analytic, except a finite set of complex numbers. Since $F$ is real on the torus, we have $G(e^{it})=\overline{F(e^{it})}=F(e^{it})$ for every $t\in\R$ and the analytic extension principle implies that 
\begin{eqnarray*}
  F\left(z\right)=\frac{P\left(z\right)}{Q\left(z\right)}=G(z)=\frac{\overline{P\left(\frac{1}{\overline{z}}\right)}}{\overline{Q\left(\frac{1}{\overline{z}}\right)}},
\end{eqnarray*}
thus 
\begin{equation}
P\left(z\right)\overline{Q\left(\frac{1}{\overline{z}}\right)}=\overline{P\left(\frac{1}{\overline{z}}\right)}Q\left(z\right)
\end{equation}
except for a finite set in $\C$. (2.1) implies that if  $P\left(\alpha\right)=0$, with $\alpha\neq0$ then necessarily $P\left(\frac{1}{\overline{\alpha}}\right)=0$. Then $P$ can be written as
$$P\left(z\right)=c_{1}~z^{m_{0}}~\left(z-\alpha_{1}\right)^{m_{1}}...\left(z-\alpha_{p}\right)^{m_{p}}~\left(1-\overline{\alpha_{1}}z\right)^{m_{1}}...\left(1-\overline{\alpha_{p}}z\right)^{m_{p}}$$
with a constant $c_{1}$. With the same argument as before, $Q$ can be written as
$$Q\left(z\right)=c_{2}~z^{d_{0}}~\left(z-\beta_{1}\right)^{d_{1}}...\left(z-\beta_{q}\right)^{d_{q}}~\left(1-\overline{\beta_{1}}z\right)^{d_{1}}...\left(1-\overline{\beta_{q}}z\right)^{d_{q}}$$ with a constant $c_{2}$. Since $P$ and $Q$ are coprime, we must have $m_{0}=0$ or $d_{0}=0$.

Then $$F\left(e^{it}\right)=\vert F\left(e^{it}\right)\vert=c~\vert \dfrac{P_{1}\left(e^{it}\right)}{Q_{1}\left(e^{it}\right)}\vert^{2}$$
where $P_{1}\left(z\right)=\prod_{j=1}^{p}\left(z-\alpha_{j}\right)^{m_{j}}$ and $Q_{1}\left(z\right)=\prod_{j=1}^{q}\left(z-\beta_{j}\right)^{d_{j}}$ therefore we have
$$F\left(e^{it}\right)=c~\dfrac{\prod_{j=1}^{p}\left(e^{it}-\alpha_{j}\right)^{m_{j}}\left(e^{-it}-\overline{\alpha_{j}}\right)^{m_{j}}}{\prod_{j=1}^{q}\left(e^{it}-\beta_{j}\right)^{d_{j}}\left(e^{-it}-\overline{\beta_{j}}\right)^{d_{j}}}$$ with a constant $c$.
Let $m=m_{1}+\cdots+m_{p}$, $d=d_{1}+\cdots+d_{q}$ and $\varphi(z)=z^{r}\psi(z)$ where $r=\max(0,m-d+1)$. 
Now,
\begin{eqnarray*}
\varphi\left(e^{it}\right)F\left(e^{it}\right)&=&c~e^{irt}\psi\left(e^{it}\right)\dfrac{\prod_{j=1}^{p}\left(e^{it}-\alpha_{j}\right)^{m_{j}}\left(e^{-it}-\overline{\alpha_{j}}\right)^{m_{j}}}{\prod_{j=1}^{q}\left(e^{it}-\beta_{j}\right)^{d_{j}}\left(e^{-it}-\overline{\beta_{j}}\right)^{d_{j}}}
\end{eqnarray*}

\begin{eqnarray*}
&=&c~e^{irt}\prod_{j=1}^{q}\left( \dfrac{e^{it}-\beta_{j}}{1-\overline{\beta_{j}}e^{it}}\right)^{d_{j}}\dfrac{\prod_{j=1}^{p}\left(e^{it}-\alpha_{j}\right)^{m_{j}}\left(1-\overline{\alpha_{j}}e^{it}\right)^{m_{j}}e^{-imt}}{\prod_{j=1}^{q}\left(e^{it}-\beta_{j}\right)^{d_{j}}\left(1-\overline{\beta_{j}}e^{it}\right)^{d_{j}}e^{-idt}}\\
&=&c~e^{i\left(d-m\right)t}e^{irt}~\dfrac{\prod_{j=1}^{p}\left(e^{it}-\alpha_{j}\right)^{m_{j}}\left(1-\overline{\alpha_{j}}e^{it}\right)^{m_{j}}}{\prod_{j=1}^{q}\left(1-\overline{\beta_{j}}e^{it}\right)^{2d_{j}}}\\
&=&c~e^{it\max(d-m,1)}~\dfrac{\prod_{j=1}^{p}\left(e^{it}-\alpha_{j}\right)^{m_{j}}\left(1-\overline{\alpha_{j}}e^{it}\right)^{m_{j}}}{\prod_{j=1}^{q}\left(1-\overline{\beta_{j}}e^{it}\right)^{2d_{j}}}\\
\end{eqnarray*}
which implies that  $\varphi F\in \hbox{{\rm H}\kern-0.9em\hbox{{\rm I}\ \ }}^1_{0}$. It follows from Lemma 2.2 that we have $F=\lvert f\lvert^{2}$ with a suitable $f\in H(\varphi)$. Then for any integer $k$, we get
\begin{eqnarray*}
\vert c_{k}\vert=\vert < F,e^{ikt}>\vert&=&\vert< f\overline{f},e^{ikt}>\vert\\
&=&\vert< fe^{-ikt},f>\vert\\
&=&\vert< {(\bbs(\varphi))}^{k}f,f>\vert\\
&=&\vert< {\bbs}^{k}(\varphi)f,f>\vert.
\end{eqnarray*}
Therefore$$\vert c_{k}\vert\leq \Vert f\Vert_{2}^{2}~\omega_{2}({\bbs}^{k}(\varphi))=\Vert F\Vert_{1}~\omega_{2}({\bbs}^{k}(\varphi))=c_{0}~\omega_{2}({\bbs}^{k}(\varphi)).$$
The proof is now complete.
\end{proof}

\begin{corollary}[Egerváry and Szász \cite{Eger}]
\textit{ Let $F(e^{it})=\sum_{j=-n+1}^{n-1}c_{j}e^{ijt}$ be a positive trigonometric polynomial $(n\geq2)$. Then $$\lvert c_{k}\lvert\leqslant c_{0}\cos\left( \dfrac{\pi}{\left[ \frac{n-1}{k}\right] +2 }\right)~~~for~~~1\leqslant k\leqslant n-1.$$}
\end{corollary}
\begin{proof}
We have :
\begin{eqnarray*}
F(e^{it})&=&c_{-n+1}e^{i(-n+1)t}+\dots+c_{0}+\dots+c_{n-1}e^{i(n-1)t}\\
&=& e^{(-n+1)it}\left( c_{-n+1}+\dots+c_{0}e^{i(n-1)t}+\dots+c_{n-1}e^{2i(n-1)t} \right) \\
&=&\dfrac{P(e^{it})}{Q(e^{it})}
\end{eqnarray*}
where $P(e^{it})=c_{-n+1}+\dots+c_{0}e^{i(n-1)t}+\dots+c_{n-1}e^{2i(n-1)t}$ and $Q(e^{it})=e^{i(n-1)t}$. In this case we have $m=n-1$, $d=0$ and $\varphi(z)=z^{n}$. Then Theorem 2.1 implies that 

$$\vert c_{k}\vert\leq c_{0}~\omega_{2}( {{\bbS}^{\ast}_{n}}^{k}).$$

But generally ${{\bbS}^{\ast}_{n}}^{k}$ is unitarily equivalent to an orthogonal sum of shifts of smaller dimension, the largest dimension being $s(k,n)+1$ where $s(k,n)=[\frac{n-1}{k}]$. Therefore $\omega_{2}({{\bbS}^{\ast}_{n}}^{k})=\omega_{2}({\bbS}^{\ast}_{s(k,n)+1})=\cos\dfrac{\pi}{s(k,n)+2}$. The same computation follows from \cite{Gustafson}. Finally, this implies that 
\begin{eqnarray*}
 \lvert c_{k} \lvert &\leqslant& c_{0}~\cos\dfrac{\pi}{s(k,n)+2}\\
&=&c_{0}\cos\left( \dfrac{\pi}{\left[ \frac{n-1}{k}\right] +2 }\right).
\end{eqnarray*}
\end{proof}

\begin{remark}
The bound for $c_{1}$ is due to Fejer (1915).
\end{remark}
\section{The numerical radius of the shift compression}
\subsection{Preliminaries}
The spectral decomposition of the KMS matrix 
$$K_{n}(\alpha)=
\left(
\begin{array}{cccc}
1&\alpha&\cdots&\alpha^{n-1}\\
\alpha&\ddots&\ddots &  \vdots                \\
\vdots&  \ddots               &  \ddots                      &\alpha          \\
\alpha^{n-1}&  \cdots               &   \alpha                     &        1
\end{array}
\right)=(\alpha^{\lvert r-s\lvert})^{n}_{r,s=1}
~~~~; 0\leqslant\alpha<1$$

 is very well understood in the computational sense. For this reason, these matrices are often used as test matrices. It's shown in \cite{Szego} page 69--72 that $K_{n}(\alpha)$ is a Toeplitz matrix associated with the Poisson kernel $P_{\alpha}(e^{it})=\dfrac{1-\alpha^{2}}{{\vert 1-\alpha e^{it}\vert}^{2}}$ and its eigenvalues are:$$ \lambda_{k}^{(n)}=P_{\alpha}(e^{it_{k}^{(n)}})~~,1\leq k \leq n$$
where $t_{k}^{(n)}$ are the solutions of 
\begin{equation}
 p_{n}(\cos t)=\dfrac{\sin(n+1)t-2\alpha\sin nt+\alpha^{2}\sin(n-1)t}{\sin t}=0.
\end{equation}
The expression $p_{n}(\cos t)$ is a polynomial of degree $n$ in $\cos t$ and it has $n$ real distinct zeros $\cos t_{k}^{(n)}$ for $1\leqslant k\leqslant n$ where :
$$0< t_{1}^{(n)} <  t_{2}^{(n)} < t_{3}^{(n)} < \cdots  < t_{n}^{(n)}< \pi~.$$
This implies that $$\dfrac{1+\alpha}{1-\alpha} >\lambda_{1}^{(n)} > \lambda_{2}^{(n)} > \lambda_{3}^{(n)} > \cdots > \lambda_{n}^{(n)} > \dfrac{1-\alpha}{1+\alpha} ~. $$
The evaluation of the zeros $t_{k}^{(n)}$ in explicit terms seems to be out of end. However, it is easy to show that they are separated by the quantities $x_{k}=\dfrac{k\pi}{n+1},~1\leq k \leq n$. Indeed, for $1\leqslant k\leqslant n$  $$p_{n}(\cos x_{k})=(-1)^{k}2\alpha(1-\alpha\cos x_{k})$$ and  $$ sgn~p_{n}(\cos x_{k})=(-1)^{k}~.$$
Also we see by direct substitution that the latter equation holds for $k=0$, so that $$0< t_{1}^{(n)} \leqslant x_{1}<  t_{2}^{(n)}\leqslant x_{2}  < \cdots  < t_{n}^{(n)}\leqslant x_{n}< \pi ~.$$

\begin{remark}
In the case where $\alpha=0$ in (3.1) we have $t_{k}^{(n)}=x_{k}$. 
\end{remark}
\ In the next proposition we give a better lower and upper bound for $t_{1}^{n}$.

\begin{proposition} 
\textit{For each integer $n\geq2$;
\begin{equation}
\frac{2}{n+1}\arccos(\alpha)\leqslant t_{1}^{(n)}\leqslant \arccos(\alpha)~.
\end{equation}}
\end{proposition}
\begin{proof}
 First, we note that 
\begin{equation}
p_{n}(t)=\tfrac{2}{\sin t}\left( \sin\tfrac{(n+1)t}{2}-\alpha \sin\tfrac{(n-1)t}{2}\right) \left( \cos\tfrac{(n+1)t}{2}-\alpha \cos\tfrac{(n-1)t}{2}\right). 
\end{equation}
Since for all  $0 < t \leqslant \frac{\pi}{n+1}$, we have 
\begin{equation}
0< \tfrac{(n-1)t}{2} < \tfrac{(n+1)t}{2}\leqslant\tfrac{\pi}{2}
\end{equation}
this implies that $$\alpha  \sin\tfrac{(n-1)t}{2}< \sin\tfrac{(n-1)t}{2}<\sin\tfrac{(n+1)t}{2}$$ then $t_{1}^{(n)}$ is zero of  
\begin{equation}
\cos\tfrac{(n+1)t}{2}=\alpha \cos\tfrac{(n-1)t}{2}.
\end{equation}
Now if we suppose that  $ t_{1}^{(n)}<\tfrac{2}{n+1}\arccos(\alpha)$ then $\tfrac{(n+1)t_{1}^{(n)}}{2}< \arccos (\alpha)$ and $
\cos\tfrac{(n+1)t_{1}^{(n)}}{2}> \alpha \geq \alpha \cos\tfrac{(n-1)t_{1}^{(n)}}{2}$ which contradicts(3.5).\\ Hence $t_{1}^{(n)}\geq \tfrac{2}{n+1}\arccos(\alpha)$ holds.

\ From (3.5), we have $$\cos t ~ \cos\tfrac{(n-1)t}{2} - \sin t~ \sin\tfrac{(n-1)t}{2}=\alpha \cos\tfrac{(n-1)t}{2}$$
which implies that $$\left( \cos t -\alpha\right)~\cos\tfrac{(n-1)t}{2}=\sin t~\sin\tfrac{(n-1)t}{2}$$
while from (3.4), ~ $\sin\tfrac{(n-1)t}{2}$~ and ~$\cos\tfrac{(n-1)t}{2}$ are both positive, therefore $ \cos t -\alpha $ is positive, which completes the proof.
\end{proof}
\begin{remark}
For a fixed $0\leqslant\alpha<1$, since $P_{\alpha}(e^{it})$ is positive and decreasing on the interval $[0,\pi]$, then it is easy to obtain the sharp lower and upper bound of the largest eigenvalues of $K_{n}(\alpha)$: $$1\leqslant\lambda_{1}^{(n)} \leqslant\frac{1-\alpha^{2}}{1-2\alpha \cos\left( \frac{2}{n+1}\arccos(\alpha)\right) +\alpha^{2}}.$$ 
Note that $\lambda_{1}^{(n)}$ is also the numerical radius of $K_{n}(\alpha)$. This is due to the fact that the norm and numerical radius of a symmetric matrix coincides with its largest eigenvalue.
\end{remark}

\ For $0\leqslant \alpha<1$, we denote by
 $$s_{n}(\alpha)=\max\left\lbrace\frac{\alpha(\cos t_{1}^{(n)}-\alpha)t}{1-2\alpha \cos t_{1}^{(n)}+\alpha^{2}};\frac{-\alpha(\cos t_{n}^{(n)}-\alpha)}{1-2\alpha \cos t_{n}^{(n)}+\alpha^{2}}\right\rbrace ;$$

$$m_{n}(\alpha)=\max\left\lbrace\frac{\lvert(1+\alpha^{2})\cos t_{1}^{(n)}-2\alpha\lvert}{1-2\alpha \cos t_{1}^{(n)}+\alpha^{2}},\frac{-(1+\alpha^{2})\cos t_{n}^{(n)}+2\alpha}{1-2\alpha \cos t_{n}^{(n)}+\alpha^{2}}\right\rbrace$$
and
$$M_{n}(\alpha)=\max\left\lbrace\frac{(1-3\alpha^{2})\cos t_{1}^{(n)}+2\alpha^{3}}{1-2\alpha \cos t_{1}^{(n)}+\alpha^{2}},\frac{-(1+\alpha^{2})\cos t_{n}^{(n)}+2\alpha}{1-2\alpha \cos t_{n}^{(n)}+\alpha^{2}}\right\rbrace .$$

\begin{proposition}
\textit{ For each integer $n\geq2$ and $0\leqslant\alpha<1$, let
$$ J_{n}(\alpha)=\left(
\begin{array}{cccc}
0        &\alpha    &\cdots  &\alpha^{n-1}\\
 \vdots        & \ddots   &\ddots  &\vdots        \\
 \vdots        &          & \ddots &\alpha       \\
  0       &   \dots       &    \dots    & 0
\end{array}
\right);  $$then we have
 $$\omega_{2}( J_{n}(\alpha))=\omega_{2}(\bbg( J_{n}(\alpha)))=s_{n}(\alpha).$$}
\end{proposition}
\begin{proof}
First observe that
\begin{eqnarray*}
\omega_{2}(\bbg( J_{n}(\alpha)))&=&\sup_{u=(u_{0},\cdots,u_{n-1})\in \C^{n},\lVert u\lVert=1}\lvert<\bbg( J_{n}(\alpha))u,u>\lvert\\
&=&\frac{1}{2}\sup_{\sum_{l=0}^{n-1}\vert{ u_{l}\vert}^{2}=1}\Big\vert\sum_{1\leq m \neq l \leq n-1}\alpha^{\vert l-m\vert} u_{l}\overline{u_{m}}\Big\vert\\
&=&\frac{1}{2}\sup_{\sum_{l=0}^{n-1}\vert{ u_{l}\vert}^{2}=1}\sum_{1\leq m \neq l \leq n-1}\alpha^{\vert l-m\vert} \vert u_{l}\vert \vert u_{m}\vert\\
&=&\frac{1}{2}\sup_{\sum_{l=0}^{n-1}\vert{ u_{l}\vert}^{2}=1} 2 \sum_{1\leq m < l \leq n-1}\alpha^{\vert l-m\vert} \vert u_{l}\vert \vert u_{m}\vert\\
&=&\sup_{\sum_{l=0}^{n-1}\vert{ u_{l}\vert}^{2}=1}\Big\vert  \sum_{1\leq m < l \leq n-1}\alpha^{l-m}  u_{l}\overline{u_{m}}\Big\vert=\omega_{2}( J_{n}(\alpha)).
\end{eqnarray*}
 We note that $ \bbg( J_{n}(\alpha))$ is the Toeplitz matrix associated with the Toeplitz form  $$\dfrac{1}{2}(P_{\alpha}(e^{it})-1)=\dfrac{\alpha(\cos t-\alpha)}{1-2\alpha \cos t+\alpha^{2}}=g(t).$$
To complete the proof of the proposition, we can easily observe that if $a$ and $b$ are arbitrary real number and $f(x)$ a Toeplitz form with $\gamma_{k}^{n}$ as eigenvalues then the eigenvalues of $a+bf(x)$ will be $a+b\gamma_{k}^{n}$. This shows that the eigenvalues of $ \bbg(J_{n}( \alpha))$ are $$\lambda'^{(n)}_{k}=\dfrac{1}{2}(\lambda_{k}^{(n)}-1)=\dfrac{\alpha(\cos t_{k}^{(n)}-\alpha)}{1-2\alpha \cos t_{1}^{(n)}+\alpha^{2}},~ 1\leq k\leq n .$$ Now, since  $g(t)$ is decreasing on the interval $[0,\pi]$ and $\bbg( J_{n}(\alpha)$ is symmetric then
\begin{eqnarray}
\omega_{2}( J_{n}(\alpha))&=&\omega_{2}(\bbg( J_{n}(\alpha)))\nonumber\\
&=&\max \left\lbrace \vert \lambda'^{(n)}_{k}\vert ; 1\leqslant k\leqslant n\right\rbrace \nonumber\\
&=&\max \left\lbrace  \lambda'^{(n)}_{1};-\lambda'^{(n)}_{n}\right\rbrace\nonumber \\
&=&\max\left\lbrace\dfrac{\alpha(\cos t_{1}^{(n)}-\alpha)}{1-2\alpha \cos t_{1}^{(n)}+\alpha^{2}};\dfrac{-\alpha(\cos t_{n}^{(n)}-\alpha)}{1-2\alpha \cos t_{n}^{(n)}+\alpha^{2}}\right\rbrace \\
&=&s_{n}(\alpha)\nonumber
\end{eqnarray}

where (3.6) is due to the fact that $\cos t_{n}^{(n)}$ is nonpositive and by using Proposition 3.1.
\end{proof}

\begin{corollary}\textit{ For $0\leqslant\alpha< 1$, we have
$$\omega_{2}(J_{2}(\alpha))=\dfrac{\alpha}{2}~.$$}
\end{corollary}
\begin{proof} This result is known, but it is interesting to notice that this result can also easily be obtained by using Proposition 3.2. Indeed  $$p_{2}(t)=\dfrac{\sin(3t)-2\alpha \sin(2t)+\alpha^{2}\sin t}{\sin t}=4\cos^{2}t-4\alpha\cos t+\alpha^{2}-1$$ Therefore, we obtain: $\cos t_{1}^{(2)}=\dfrac{\alpha+1}{2}$ and  $\lambda'^{(2)}_{1}=\dfrac{\alpha}{2}$ . Now since $Tr(J_{2}(\alpha))=0$ then we have $\lambda'^{(2)}_{2}=-\dfrac{\alpha}{2}$. This completes the proof.
\end{proof}
\begin{corollary} \textit{For $0\leqslant\alpha< 1$, we have
 $$\omega_{2}(J_{3}(\alpha))=\dfrac{\alpha(\sqrt{\alpha^{2}+8}-3\alpha)}{4+2\alpha^{2}-2\alpha\sqrt{\alpha^{2}+8}}~.$$}
\end{corollary}
\begin{proof} We have $$p_{3}(t)=\dfrac{2}{\sin t}\big( \sin(2t)-\alpha\sin t \big) \big( \cos(2t)-\alpha \cos t\big) $$ this implies that $\cos t_{1}^{(3)}=\dfrac{\alpha+\sqrt{\alpha^{2}+8}}{4}, \cos t_{2}^{(3)}=\dfrac{\alpha}{2}$ and $\cos t_{3}^{(3)}=\dfrac{\alpha-\sqrt{\alpha^{2}+8}}{4}$ then $$\lambda'^{(3)}_{1}=\dfrac{\alpha(\cos t_{1}^{(3)}-\alpha)}{1-2\alpha \cos t_{1}^{(3)}+\alpha^{2}}=\dfrac{\alpha(\sqrt{\alpha^{2}+8}-3\alpha)}{4+2\alpha^{2}-2\alpha\sqrt{\alpha^{2}+8}}$$ and $$\lvert \lambda'^{(3)}_{3}\lvert=\dfrac{-\alpha(\cos t_{3}^{(3)}-\alpha)}{1-2\alpha \cos t_{3}^{(3)}+\alpha^{2}}=\dfrac{\alpha(\sqrt{\alpha^{2}+8}+3\alpha)}{4+2\alpha^{2}+2\alpha\sqrt{\alpha^{2}+8}}.$$ The proof is complete.
\end{proof}
\begin{corollary} \textit{For each integer $n \geq 4$ and $\alpha\leqslant\sqrt{\cos\frac{2\pi}{n+1}} $, we have$$\omega_{2}(J_{n}(\alpha))=\dfrac{\alpha(\cos t_{1}^{(n)}-\alpha)}{1-2\alpha \cos t_{1}^{(n)}+\alpha^{2}}~.$$}
\end{corollary}
\begin{proof} It follows from the hypothesis of Corollary 3.5 that
\begin{equation} 
t_{1}^{(n)}\leq \frac{\pi}{n+1}\leq \arccos\sqrt{\frac{1+\alpha^{2}}{2}}\leq \arccos(\alpha)\leqslant t_{n}^{(n)}\leqslant  \frac{n\pi}{n+1}
\end{equation}
 then $$\lambda'^{(n)}_{1}=\dfrac{\alpha(\cos t_{1}^{(n)}-\alpha)}{1-2\alpha \cos t_{1}^{(n)}+\alpha^{2}}$$ and
$$\vert\lambda'^{(n)}_{n}\vert=-\lambda'^{(n)}_{n}=\dfrac{-\alpha(\cos t_{n}^{(n)}-\alpha)}{1-2\alpha \cos t_{n}^{(n)}+\alpha^{2}}$$
In view of the inequality (3.7), $g(\arccos \alpha)=0$ and the fact that $g$ is decreasing in $[0,\pi]$, it suffices to prove that $g(\frac{\pi}{n+1}) \geq \vert g(\frac{n\pi}{n+1})\vert$. We have
\begin{eqnarray*}
 g(\frac{\pi}{n+1})-\vert g(\frac{n\pi}{n+1})\vert&=&\dfrac{\alpha(\cos \frac{\pi}{n+1}-\alpha) }{1-2\alpha \cos \frac{\pi}{n+1}+\alpha^{2}}- \dfrac{\alpha(\cos \frac{\pi}{n+1}+\alpha) }{1+2\alpha \cos \frac{\pi}{n+1}+\alpha^{2}}\\
&=&\dfrac{2\alpha^{2}(2\cos^{2}\frac{\pi}{n+1}-\alpha^{2}-1)}{(1-2\alpha \cos \frac{\pi}{n+1}+\alpha^{2})(1+2\alpha \cos \frac{\pi}{n+1}+\alpha^{2})}\\
&=&\dfrac{2\alpha^{2}(\cos\frac{2\pi}{n+1}-\alpha^{2})}{(1-2\alpha \cos \frac{\pi}{n+1}+\alpha^{2})(1+2\alpha \cos \frac{\pi}{n+1}+\alpha^{2})}\\
&\geq&0.
\end{eqnarray*}
The result follows.
\end{proof}
\subsection{The numerical radius of the compressed shift} In this section, we will focus on the particular case where
$$\phi(z)=\left( \dfrac{z-\alpha}{1-\overline{\alpha}z}\right)^n.$$
First, we notice some properties for the general case where $\phi$ is a finite Blashke product : $\phi(z)= \prod_{j=1}^{n}~\dfrac{z-\alpha_{j}}{1-\overline{\alpha_{j}}z}$. For each $\lambda$ in the unit disc, we define the evaluation functional $k_{\lambda} \in \H2$ by the requirement that $f(\lambda)=<f,k_{\lambda}> $. Thus$$k_{\lambda}(z)=\dfrac{1}{1-\overline{\lambda}z}$$ and  $\left\lbrace e_{1}, \dots,e_{n} \right\rbrace $ the collection of functions of $H(\phi)$ defined as follows :
\begin{equation*}
e_{1}(z)= \left(1-\vert\alpha_{1}\vert^{2}\right)^{\frac{1}{2}}  ~\dfrac{1}{1-\overline{\alpha_{1}}z}
\end{equation*}
 and
$$e_{k}(z)= \left(1-\vert\alpha_{k}\vert^{2}\right)^{\frac{1}{2}}~\dfrac{1}{1-\overline{\alpha_{k}}z}~~ \prod_{j=1}^{k-1}\dfrac{z-\alpha_{j}}{1-\overline{\alpha_{j}}z}$$ for any $k=2,...,n$.

It is known that $\left\lbrace e_{1}, \dots,e_{n} \right\rbrace $ is an orthonormal basis of $H(\phi)$ and with respect to this basis the matrix of $\bbs(\phi)$ is given by $\left[ a_{lk} \right]$, where
 $$ a_{lk}=\left\{
    \begin{array}{ll}
         \overline{\alpha_{l}} & \mbox{if } l=k \\ 
         \sigma_{l}\sigma_{l+1}& \mbox{if } k=l+1 \\ 
         \sigma_{l}\sigma_{k} \prod_{j=l+1}^{k-1} (-\alpha_{j}) & \mbox{if } k>l+1\\
        0 & \mbox{unless}
    \end{array}
\right.$$
and $ \sigma_{k}=\left(1-\vert\alpha_{k}\vert^{2}\right)^{\frac{1}{2}}$, for each $1\leqslant k \leqslant n$.

Indeed, for $k>l+1$, we have
\begin{eqnarray*}
<{\bbs}(\phi)e_{k},e_{l}>&=&\sigma_{k}\sigma_{l}\int_{0}^{2\pi}\dfrac{1}{1-\overline{\alpha_{k}}e^{i\theta}}~~\dfrac{e^{-i\theta}}{1-\alpha_{l}e^{-i\theta}}~\prod_{j=l}^{k-1}\dfrac{e^{i\theta}-\alpha_{j}}{1-\overline{\alpha_{j}}e^{i\theta}}\frac{d\theta}{2\pi}\\
&=&\sigma_{k}\sigma_{l}\int_{0}^{2\pi}\dfrac{1}{1-\overline{\alpha_{k}}e^{i\theta}}~~\dfrac{e^{-i\theta}}{1-\alpha_{l}e^{-i\theta}}~\dfrac{e^{i\theta}-\alpha_{l}}{1-\overline{\alpha_{l}}e^{i\theta}}~~\prod_{j=l+1}^{k-1}\dfrac{e^{i\theta}-\alpha_{j}}{1-\overline{\alpha_{j}}e^{i\theta}}\frac{d\theta}{2\pi}\\
&=&\sigma_{k}\sigma_{l}\int_{0}^{2\pi}\dfrac{1}{1-\overline{\alpha_{k}}e^{i\theta}}~~\dfrac{1}{1-\overline{\alpha_{l}}e^{i\theta}} ~~\prod_{j=l+1}^{k-1}\dfrac{e^{i\theta}-\alpha_{j}}{1-\overline{\alpha_{j}}e^{i\theta}}\frac{d\theta}{2\pi}\\
&=&\sigma_{k}\sigma_{l}\prod_{j=l+1}^{k-1} (-\alpha_{j}).
\end{eqnarray*}
\ Using the same scheme as before,we prove easily that $<{\bbs(\phi)}e_{k+1},e_{k}>=\sigma_{k}\sigma_{k+1}$ and  $<{\bbs}(e_{k}),e_{l}>=0$ if $k<l$.

\ Finally
\begin{eqnarray*}
<{\bbs}(\phi)e_{k},e_{k}>&=&\sigma_{k}^{2}\int_{0}^{2\pi}\dfrac{1}{1-\overline{\alpha_{k}}e^{i\theta}}\dfrac{e^{-i\theta}}{1-\alpha_{k}e^{-i\theta}}\frac{d\theta}{2\pi}\\
&=&\sigma_{k}^{2}< K_{\alpha_{k}},\dfrac{z}{1-\overline{\alpha_{k}}z}>\\
&=&\overline{\alpha_{k}}
\end{eqnarray*}
In the sequel of the paper, $\phi$ denotes the finite Blashke product with unique zero $\alpha$:
$$
\phi(z)=\left( \dfrac{z-\alpha}{1-\overline{\alpha}z}\right)^n.
$$
$\bbs(\phi)$ gets the following matricial representation:
$$
\left(
\begin{array}{cccccc}
\overline{\alpha}&\sigma           &-\sigma\alpha &\cdots &\cdots           &\sigma(-\alpha)^{n-2} \\
0                &\overline{\alpha}& \sigma       &\ddots &                 &\vdots \\
\vdots           &       \ddots    & \ddots       &\ddots &\ddots           &\vdots\\
\vdots           &                 & \ddots       &\ddots &\ddots           &-\sigma\alpha\\
\vdots           &                 &              &\ddots &\overline{\alpha}&\sigma\\
0                &       \dots     &  \dots       &\dots  &         0       &\overline{\alpha}
\end{array}
\right)
$$
where $ \sigma=1-\vert \alpha\vert^{2}$.

\begin{proposition} \textit{For $0\leqslant\alpha<1$, we have 
$$\omega_{2}(\bbg(\bbs(\phi)))=m_{n}(\alpha).$$}
\end{proposition}
\begin{proof} First, notice that where $ \alpha=0$, then
$$\bbg(\bbs(\phi))=\dfrac{1}{2}
\left(
\begin{array}{ccccc}
0     & 1   &  0   & 0    & \dots\\
1     & 0   &  1   & 0    &   \dots               \\
0     & 1   &  0   & 1    &\dots          \\
0     &  0  &   1  & 0    &\dots         \\
\dots &\dots&\dots & \dots &\dots

\end{array}
\right)
.$$
In this case the eigenvalues are $\cos\dfrac{k\pi}{n+1}$, for $k=1,\dots,n$. For the proof there are many references, we refer the reader for example to \cite{Szego} page 67 or \cite{Bottcher} page 35, therefore $\omega_{2}(\bbg(\bbs(\phi)))=\cos\dfrac{\pi}{n+1}=m_{n}(0)$. Then we can limit our study to the case $\alpha\neq 0 $. Now, notice that 
$$\bbg(\bbs(\phi))=\dfrac{1-\alpha^{2}}{2\alpha}
\left(
\begin{array}{cccc}
-\dfrac{2\alpha^{2}}{1-\alpha^{2}}&\alpha&\cdots&\alpha^{n-1}\\
\alpha&\ddots&\ddots &  \vdots                \\
\vdots&  \ddots               &  \ddots                      &\alpha          \\
\alpha^{n-1}&  \cdots               &   \alpha                     &   -\dfrac{2\alpha^{2}}{1-\alpha^{2}}     
\end{array}
\right).
$$
Here $\bbg(\bbs(\phi))$ is the Toeplitz matrix associated with the Toeplitz form: $$\dfrac{1-\alpha^{2}}{2\alpha}(P_{\alpha}(e^{it})-1-\dfrac{2\alpha^{2}}{1-\alpha^{2}})=\dfrac{(1+\alpha^{2})\cos t-2\alpha}{1-2\alpha \cos t+\alpha^{2}}=h(t).$$ 
Since $h(t)$ is monotonic on $[0,\pi]$, thus with the same argument that in the proof of Proposition 3.2, we may assume that: 
\begin{eqnarray*}
\omega_{2}(\bbg(\bbs(\phi)))&=&
\max\left\lbrace\tfrac{\lvert(1+\alpha^{2})\cos t_{1}^{n}-2\alpha\lvert}{1-2\alpha \cos t_{1}^{n}+\alpha^{2}},\tfrac{-(1+\alpha^{2})\cos t_{n}^{n}+2\alpha}{1-2\alpha \cos t_{n}^{n}+\alpha^{2}}\right\rbrace\\
&=&m_{n}(\alpha).
\end{eqnarray*}
This completes the proof.
\end{proof}
\newpage
\begin{theorem}\textit{ Let $\phi(z)=\left( \dfrac{z-\alpha}{1-\overline{\alpha}z}\right)^n $ with $\alpha\in\C$ and $\lvert\alpha\lvert<1$.}
\begin{enumerate}
 \item  \textit{The numerical radius of $\bbs(\phi)$ is independent from the argument of $\alpha$ and for $0 \leqslant \alpha<1 $ the numerical range of  $\bbs(\phi)$ is symmetric with respect to the real axis.}

\item  \textit{For $n\geq2$ , we have $$m_{n}(\lvert\alpha\lvert)\leqslant\omega_{2}(\bbs(\phi))\leqslant M_{n}(\lvert\alpha\lvert).$$}
\end{enumerate}
\end{theorem}
\begin{proof}
For $\alpha\neq0 $, and for $t=\arg(\alpha)$, we have:
\begin{eqnarray}
\omega_{2}(\bbs(\phi))&=&\sup_{\Vert u \Vert_{2}=1}\vert<\bbs(\phi)u,u> \vert \nonumber\\
&=&\sup_{\sum_{l=0}^{n-1}\vert{ u_{l}\vert}^{2}=1}\Big\vert \overline {\alpha}-\frac{1-{\vert \alpha \vert}^{2}}{\alpha}\sum_{1\leq m < l \leq n-1}\left(-\alpha\right)^{l-m}u_{l}\overline{u_{m}}\Big\vert \nonumber\\
&=&\sup_{\sum_{l=0}^{n-1}\vert{ u_{l}\vert}^{2}=1}\Big\vert \vert\alpha\vert-\frac{1-{\vert \alpha \vert}^{2}}{\vert\alpha\vert}\sum_{1\leq m < l \leq n-1}\left(-\alpha\right)^{l-m}u_{l}\overline{u_{m}}\Big\vert \nonumber\\
&=&\sup_{\sum_{l=0}^{n-1}\vert{ u_{l}\vert}^{2}=1}\Big\vert \vert\alpha\vert-\frac{1-{\vert \alpha \vert}^{2}}{\vert\alpha\vert}\sum_{1\leq m < l \leq n-1}\left(-\lvert\alpha\lvert\right)^{l-m}e^{ilt}u_{l}\overline{e^{imt}u_{m}}\Big\vert \nonumber\\
&=&\sup_{\sum_{l=0}^{n-1}\vert{ v_{l}\vert}^{2}=1}\Big\vert \vert\alpha\vert-\frac{1-{\vert \alpha \vert}^{2}}{\vert\alpha\vert}\sum_{1\leq m < l \leq n-1}\left( -\vert\alpha\vert\right) ^{l-m}v_{l}\overline{v_{m}}\Big\vert .
\end{eqnarray}
The last equality implies that the numerical radius of $\bbs(\phi)$ is independant from the argument of $\alpha$. Hence we can suppose that $0<\alpha<1$. Now assume that $z$ is in $W(\bbs(\phi))$, then there is $u=(u_{0},\dots,u_{n-1})$ a unit vector in $\C^{n}$ such that 
\begin{eqnarray*}
z&=&<\bbs(\phi)u,u>\\
&=&\alpha-\frac{1-{ \alpha }^{2}}{\alpha}\sum_{1\leq m < l \leq n-1}\left( -\alpha\right) ^{l-m}u_{l}\overline{u_{m}}
\end{eqnarray*}
and 
\begin{eqnarray*}
 \overline{z}&=&\alpha-\frac{1-{ \alpha }^{2}}{\alpha}\sum_{1\leq m < l \leq n-1}\left( -\alpha\right)^{l-m}u_{m}\overline{u_{l}}\\
&=&<\bbs(\phi)\overline{u},\overline{u}>.
\end{eqnarray*}
This implies that $\overline{z}$ is in $W(\bbs(\phi))$ and $(1)$ holds.

As remarked before, we can restrict our study to the case where $0<\alpha<1$ and from (3.8) we have:
\begin{eqnarray*}
 \omega_{2}(\bbs(\phi)) &\leq& \alpha+\frac{1-{ \alpha }^{2}}{\alpha}\sup_{\sum_{l=0}^{n-1}\vert{ u_{l}\vert}^{2}=1}\Big\vert\sum_{1\leq m < l \leq n-1}\alpha^{l-m}u_{l}\overline{u_{m}}\Big\vert\\
&=&\alpha+\frac{1-{\alpha }^{2}}{\alpha}\omega_{2}(J_{n}(\alpha))\\
&=&M_{n}(\alpha).
\end{eqnarray*}
On the other hand, it is obvious to note that $\omega_{2}(\bbg T)\leqslant \omega_{2}(T)$ for each bounded operator $T$, an application of the Proposition 3.6 completes the proof of $(2)$.
\end{proof}
\begin{corollary}
 \textit{For $n=2$ and $0 \leqslant \alpha<1 $; we have: $$ \omega_{2}(\bbs(\phi))=\dfrac{1+2\alpha-\alpha^{2}}{2}~~.$$}
\end{corollary}
\begin{corollary}
\textit{For each integer $n\geq2$;$$\omega_{2}({\bbS}^{\ast}_{n})=\cos\frac{\pi}{n+1}~~.$$}
\end{corollary}
In the last two corollaries, the results are due to the fact that both quantities $m_{n}$ and $M_{n}$ coincide.

\underline{Acknowledgements:} The author would like to express his gratitude to Gilles Cassier for his help and his good advices 
and Alfonso Montes Rodriguez as well for advices in the translation of this work..

\end{document}